\documentclass[12pt]{amsproc}
\usepackage{hyperref}
\usepackage{geometry}
\usepackage{graphicx}
\usepackage{amsthm}
\usepackage{amssymb}

\theoremstyle{plain}
\newtheorem{thm}{Theorem}[section]
\newtheorem{cor}[thm]{Corollary}

\newtheorem{prop}[thm]{Proposition}
\newtheorem{defn}[thm]{Definition}
\newtheorem{exa}[thm]{Example}
\newtheorem{rem}[thm]{Remark}

\begin{document}

\title{Graded Almost Prime Ideals over Non-Commutative Graded Rings}

\author{Jenan \textsc{Shtayat}}
\author{Rashid \textsc{Abu-Dawwas}}
\author{Ghadeer \textsc{Bani Issa}}
\address{Department of Mathematics, Yarmouk University, Irbid, Jordan}
\email{jenan@yu.edu.jo}
\email{rrashid@yu.edu.jo}
\email{gadeerbanyissa5@gmail.com}



\begin{abstract}
The purpose of this article is to define and examine graded almost prime ideals over a non-commutative graded ring, and consider some cases where all graded right ideals of a non-commutative graded ring are graded almost prime.
\end{abstract}

\maketitle

\section{Introduction}

Let $G$ be a group. Then a ring $R$ is called a $G$-graded ring if $R=\displaystyle\bigoplus_{g\in G} R_{g}$ with the property $R_{g}R_{h}\subseteq R_{gh}$ for all $g,h \in G$, where $R_g$ is an additive subgroup of $R$ for all $g\in G$. The elements of $R_{g}$ are called homogeneous of degree $g$. If $x\in R$, then $x$ can be written uniquely as $\displaystyle\sum_{g\in G} x_{g}$, where $x_{g}$ is the component of $x$ in $R_{g}$. The set of all homogeneous elements of $R$ is $\displaystyle\bigcup_{g\in G} R_{g}$ and is denoted by $h(R)$. The component $R_{e}$ is a subring of $R$, and if $R$ has unity $1$, then $1\in R_{e}$. For more terminology, see \cite{Nastasescu}. Let $R$ be a $G$-graded and $P$ be a right ideal of $R$. Then $P$ is said to be a graded right ideal if $P=\displaystyle\bigoplus_{g\in G}(P\bigcap R_{g})$, i.e, for $x\in P$, $x_{g}\in P$ for all $g\in G$. A right deal of a graded ring is not necessary to be graded right ideal, see (\cite{Dawwas Bataineh Muanger}, Example 1).

A proper ideal $P$ of a commutative ring $R$ is said to be prime if $xy\in P$ implies $x\in P$ or $y\in P$ for all $x, y\in R$. The concept of graded prime ideals over commutative graded rings was introduced in \cite{Refai Hailat}. A proper graded ideal $P$ of a commutative graded ring $R$ is said to be graded prime if $xy\in P$ implies $x\in P$ or $y\in P$ for all $x, y\in h(R)$. Assume that $P$ is a prime ideal of a commutative ring $R$. If $R$ is graded and $P$ is a graded ideal of $R$, then it will be easy to see that $P$ is a graded prime ideal of $R$. On the other hand, if $P$ is a graded prime ideal of $R$, then $P$ is not necessary to be prime ideal of $R$, see the following example:

\begin{exa}\label{1}
Consider $R=\mathbb{Z}[i]$ and $G=\mathbb{Z}_{2}$. Then $R$ is $G$-graded by
$R_{0}=\mathbb{Z}$ and $R_{1}=i\mathbb{Z}$. Consider the graded ideal $P=2R$
of $R$. We show that $P$ is a graded prime ideal of $R$. Let $xy\in P$ for
some $x, y\in h(R)$.

\underline{Case (1)}: Assume that $x,y\in R_{0}$. In this case, $x,y\in
\mathbb{Z}$ such that $2$ divides $xy$, and then either $2$ divides $x$ or
$2$ divides $y$ as $2$ is a prime number, which implies that $x\in P$ or
$y\in P$.

\underline{Case (2)}: Assume that $x,y\in R_{1}$. In this case, $x=ia$ and
$y=ib$ for some $a,b\in\mathbb{Z}$ such that $2$ divides $xy=-ab$, and then
$2$ divides $a$ or $2$ divides $b$ in $%
\mathbb{Z}
$, which implies that $2$ divides $x=ia$ or $2$ divides $y=ib$ in $R$.\ Then
we have that $x\in P$ or $y\in P$.

\underline{Case (3)}:\ Assume that $x\in R_{0}$ and $y\in R_{1}$. In this
case, $x\in\mathbb{Z}$ and $y=ib$ for some $b\in\mathbb{Z}$ such that $2$
divides $xy=ixb$ in $R$, that is $ixb=2(\alpha+i\beta)$ for some
$\alpha,\beta\in\mathbb{Z}$.\ Then we obtain $xb=2\beta$, that is $2$
divides $xb$ in $\mathbb{Z}$, and again $2$ divides $x$ or $2$ divides $b$,
which implies that $2$ divides $x$ or $2$ divides $y=ib$ in $R$. Thus, $x\in
P$ or $y\in P$.

One can similarly show that $x\in P\ $or $y\in P\ $in other cases. So, $P$ is
a graded prime ideal of $R$. On the other hand, $P$ is not a prime ideal of $R$ since
$(1-i)(1+i)\in P$, $(1-i)\notin P$ and $(1+i)\notin P$.
\end{exa}

Almost prime ideals over commutative rings have been appeared for the first time in \cite{Bhatwadekar}. A proper ideal $P$ of a commutative ring $R$ is said to be almost prime if $xy\in P-P^{2}$ implies $x\in P$ or $y\in P$ for all $x, y\in R$. The concept of graded almost prime ideals over commutative graded rings was introduced in \cite{Jaber Bataineh Khashan}. A proper graded ideal $P$ of a commutative graded ring $R$ is said to be graded almost prime if $xy\in P-P^{2}$ implies $x\in P$ or $y\in P$ for all $x, y\in h(R)$. Assume that $P$ is an almost prime ideal of a commutative ring $R$. If $R$ is graded and $P$ is a graded ideal of $R$, then it will be easy to see that $P$ is a graded almost prime ideal of $R$. On the other hand, if $P$ is a graded almost prime ideal of $R$, then $P$ is not necessary to be almost prime ideal of $R$, see the following example:

\begin{exa} Consider $R=\mathbb{Z}[i]$ and $G=\mathbb{Z}_{2}$. Then $R$ is $G$-graded by
$R_{0}=\mathbb{Z}$ and $R_{1}=i\mathbb{Z}$. Consider the graded ideal $P=2R$
of $R$. Clearly, $P$ is a graded almost prime ideal of $R$ since it is a graded prime ideal by Example \ref{1}. Indeed, $P$ is not an almost prime ideal of $R$ since $(1-i)(1+i)\in P-P^{2}$, $(1-i)\notin P$ and $(1+i)\notin P$.
\end{exa}

Clearly, every graded prime ideal over a commutative graded ring is graded almost prime. However, the next example shows that converse is not true in general:

\begin{exa} Consider $R=\mathbb{Z}_{12}[i]$ and $G=\mathbb{Z}_{2}$. Then $R$ is $G$-graded by $R_{0}=\mathbb{Z}_{12}$ and $R_{1}=i\mathbb{Z}_{12}$. Consider the graded ideal $P=4R$ of $R$. Clearly, $P$ is a graded almost prime ideal of $R$ since $P^{2}=P$. Indeed, $P$ is not a graded prime ideal of $R$ since $2.2\in P$ and $2\notin P$.
\end{exa}

Graded prime ideals over non-commutative graded rings have been defined and examined in \cite{Dawwas Bataineh Muanger}. A proper graded right ideal $P$ of $R$ is called graded prime if $XY\subseteq P$ implies $X\subseteq P$ or $Y\subseteq P$ for all graded right ideals $X, Y$ of $R$. If $R$ has unity, then the previous definition is equivalent to the following: a graded right ideal $P$ of $R$ is called graded prime if $xRy\subseteq P$ implies $x\in P$ or $y\in P$ for all $x, y\in h(R)$. Graded weakly prime ideals over non-commutative graded rings have been introduced and investigated in \cite{Alshehry Dawwas}. A proper graded right ideal $P$ of $R$ is called graded weakly prime if $0\neq XY\subseteq P$ implies $X\subseteq P$ or $Y\subseteq P$, for all graded right ideals $X, Y$ of $R$. The purpose of this article is following \cite{ABOUHALAKA FINDIK} to define and examine graded almost prime ideals over a non-commutative graded ring. We propose the following definition: a proper graded right ideal $P$ of $R$ is called graded almost prime if $XY\subseteq P$ and $XY\nsubseteq P^{2}$ imply $X\subseteq P$ or $Y\subseteq P$ for all graded right ideals $X, Y$ of $R$. A speedy attention provides that our definition and the notion of graded almost prime ideals over commutative graded rings with unity are equivalent. Nevertheless, the definitions disagree over non-commutative graded rings. We show in Theorem \ref{Theorem 2.5} that $I$ is a graded almost prime ideal in a non-commutative graded ring $R$ with unity if and only if $xRy\subseteq I$ and $xRy\nsubseteq I^{2}$ imply
either $x\in I$ or $y\in I$ for all $x, y\in h(R)$. Among several results, we prove that if $R$ is a graded ring with unity and $P$ is a graded ideal of $R$, then $P$ is a graded almost prime right ideal of $R$ if and only if $P$ is a graded almost prime ideal of $R$ (Proposition \ref{Proposition 2.2}). We show that if $R$ is a graded ring with unity and $I$ is a graded right ideal of $R$ with $(I^{2}:I)\subseteq I$, then $I$ is a graded prime right ideal of $R$ if and only if $I$ is a graded almost prime right ideal of $R$ (Theorem \ref{Theorem 2.6}). We prove that if $I$ is a graded ideal of $R$, then $I$ is a graded almost prime right ideal of $R$ if and only if $I/I^{2}$ is a graded weakly prime right ideal of $R/I^{2}$ (Theorem \ref{Theorem 2.11}). In Theorem \ref{Theorem 2.12}, Corollary \ref{Corollary 2.13}, Theorem \ref{Theorem 2.14} and Corollary \ref{Corollary 2.15}, we study graded almost prime right ideals over graded homomorphism. In Theorem \ref{Theorem 2.16}, we examine graded almost prime right ideals over graded quotient rings. Finally, we consider some cases where all graded right ideals of a non-commutative graded ring are graded almost prime. Throughout this article all rings are associative, non-commutative, and without unity unless stated otherwise, and by ideal we mean a proper two sided ideal.

\section{Graded Almost Prime Right Ideals}

In this section, we introduce and study the concept of graded almost prime ideals over non-commutative graded rings.

\begin{defn} Let $R$ be a graded ring and $P$ be a proper graded right ideal of $R$. Then $P$ is said to be graded almost prime if $XY\subseteq P$ and $XY\nsubseteq P^{2}$ imply $X\subseteq P$ or $Y\subseteq P$ for all graded right ideals $X, Y$ of $R$.
\end{defn}

Clearly, every graded prime right ideal is graded almost prime. Indeed, $P=\{0\}$ is a graded almost prime right ideal in any graded ring $R$ that is not necessary to be graded prime. Also, the next examples show that the converse is not true in general:

\begin{exa}\label{2} Consider the non-commutative ring $R=\left\{0, x, y, z\right\}$ under the operations:
\begin{center}
$\begin{array}{ccccc}
  + & 0 & x & y & z \\
  0 & 0 & x & y & z \\
  x & x & 0 & z & y \\
  y & y & z & 0 & x \\
  z & z & y & x & 0
\end{array}$
\hspace{2 cm}
$\begin{array}{ccccc}
   . & 0 & x & y & z \\
   0 & 0 & 0 & 0 & 0 \\
   x & 0 & x & x & 0 \\
   y & 0 & y & y & 0 \\
   z & 0 & z & z & 0
 \end{array}$
 \end{center}
 The only additive subgroups of $R$ are $\{0\}$, $\{0, x\}$, $\{0, y\}$ and $\{0, z\}$. Let $R$ be $G$-graded for some group $G$ and assume that $g\in G-\{e\}$. Then $R_{g}=\{0\}$, $\{0, x\}$, $\{0, y\}$ or $\{0, z\}$. If $R_{g}=\{0, x\}$, then $x=x.x\in R_{g}R_{g}\subseteq R_{g^{2}}$, and then $0\neq x\in R_{g}\bigcap R_{g^{2}}$, which implies that $g=g^{2}$, and hence $g=e$, a contradiction. Similarly, if $R_{g}=\{0, y\}$, then $g=e$, a contradiction. Hence, for $g\neq e$, $R_{g}=\{0\}$ or $\{0, z\}$, and $R_{e}=\{0\}$, $\{0, x\}$ or $\{0, y\}$. So, if we take $G=\mathbb{Z}_{3}$, then we can choose $R$ to be $G$-graded by $R_{0}=\{0, x\}$, $R_{1}=\{0, z\}$ and $R_{2}=\{0\}$. The only right ideals of $R$ are $\{0\}$, $P=\{0, x\}$, $I=\{0, y\}$ and $J=\{0, z\}$. Note that, $I$ is not a graded right ideal of $R$ since $y\in I$ and $y=y_{0}+y_{1}$ with $y_{0}=x, y_{1}=z\notin I$. So, $\{0\}$, $P$ and $J$ are the only graded right ideals of $R$. Since $P^{2}=P$, $P$ is a graded almost prime right ideal of $R$. On the other hand, $P$ is not a graded prime right ideal of $R$ since $J.J=\{0\}\subseteq P$ and $J\nsubseteq P$.
\end{exa}

As an application on Example \ref{2}, we introduce the following example:

\begin{exa}\label{3} Consider the subring of $M_{2}(\mathbb{Z}_{2})$:
\begin{center}
$R=\left\{\left(\begin{array}{cc}
                                       0 & 0 \\
                                       0 & 0
                                     \end{array}
\right), A=\left(\begin{array}{cc}
                 1 & 1 \\
                 0 & 0
               \end{array}
\right), B=\left(\begin{array}{cc}
                 0 & 0 \\
                 1 & 1
               \end{array}
\right), C=\left(\begin{array}{cc}
                 1 & 1 \\
                 1 & 1
               \end{array}
\right)\right\}$
\end{center}
and $G=\mathbb{Z}_{3}$. Depending on the discussion in Example \ref{2}, $R$ is $G$-graded by
\begin{center}
$R_{0}=\left\{\left(\begin{array}{cc}
                                       0 & 0 \\
                                       0 & 0
                                     \end{array}
\right), A\right\}$, $R_{1}=\left\{\left(\begin{array}{cc}
                                       0 & 0 \\
                                       0 & 0
                                     \end{array}
\right), C\right\}$ and $R_{2}=\left\{\left(\begin{array}{cc}
                                       0 & 0 \\
                                       0 & 0
                                     \end{array}
\right)\right\}$
\end{center}
with $P=\left\{\left(\begin{array}{cc}
                                       0 & 0 \\
                                       0 & 0
                                     \end{array}
\right), A\right\}$ is a graded almost prime right ideal of $R$ which is not graded prime.
\end{exa}

\begin{exa}\label{4} Consider the ring with unity $R=\left\{\left(\begin{array}{cc}
                                       a & b \\
                                       0 & c
                                     \end{array}
\right): a, b, c\in \mathbb{R}\right\}$ and $G=\mathbb{Z}_{4}$. Then $R$ is $G$-graded by

$R_{0}=\left\{\left(\begin{array}{cc}
                                       a & 0 \\
                                       0 & c
                                     \end{array}
\right): a, c\in \mathbb{R}\right\}$, $R_{2}=\left\{\left(\begin{array}{cc}
                                       0 & b \\
                                       0 & 0
                                     \end{array}
\right): b\in \mathbb{R}\right\}$ and $R_{1}=R_{3}=\left\{\left(\begin{array}{cc}
                                       0 & 0 \\
                                       0 & 0
                                     \end{array}
\right)\right\}$. Indeed, $P=\left\{\left(\begin{array}{cc}
                                       0 & 0 \\
                                       0 & c
                                     \end{array}
\right): c\in \mathbb{R}\right\}$ is a graded almost prime right ideal of $R$ which is not graded prime since $A=\left(\begin{array}{cc}
                                       0 & 1 \\
                                       0 & 0
                                     \end{array}
\right), B=\left(\begin{array}{cc}
                                       0 & 2 \\
                                       0 & 0
                                     \end{array}
\right)\in h(R)$ with $ARB=\left\{\left(\begin{array}{cc}
                                       0 & 0 \\
                                       0 & 0
                                     \end{array}
\right)\right\}\subseteq P$ and $A, B\notin P$.
\end{exa}

\begin{prop}\label{Proposition 2.2} Let $R$ be a graded ring with unity and $P$ be a graded ideal of $R$. Then $P$ is a graded almost prime right ideal of $R$ if and only if $P$ is a graded almost prime ideal of $R$.
\end{prop}

\begin{proof} Suppose that $P$ is a graded almost prime ideal of $R$. Let $X, Y$ be two graded right ideals of $R$ such that $XY\subseteq P$, and $XY\nsubseteq P^{2}$. Since $R$ has unity, $XR=X$, and then $(RX)(RY)=RXY\subseteq RP=P$, where $RX$ and $RY$ are graded ideals of $R$. Further, if $(RX)(RY)\subseteq P^{2}$, then $XY\subseteq RXY=(RX)(RY)\subseteq P^{2}$ which is a contradiction. So, $(RX)(RY)\nsubseteq P^{2}$. Hence, since $P$ is graded almost prime, we have either $X\subseteq RX\subseteq P$ or $Y\subseteq RY\subseteq P$. Thus, $P$ is a graded almost prime right ideal of $R$. The converse is clear.
\end{proof}

Let $R$ be a ring and $P, K\subseteq R$. Define $(P:K)=\left\{x\in R:Kx\subseteq P\right\}$ and $(P:_{*}K)=\left\{x\in R:xK\subseteq P\right\}$. Suppose that $R$ is a graded ring and $P, K$ are graded right ideals of $R$. Then using a similar proof to (\cite{Dawwas}, Lemma 2.19), one can prove that $(P:K)$ is a graded right ideal of $R$. Also similarly, If $P$ and $K$ are graded left ideals of $R$, then $(P:_{*}K)$ is a graded left ideal of $R$. Note that, if $P$ and $K$ are graded ideals of $R$, then so are $(P:K)$ and $(P:_{*}K)$.

\begin{thm}\label{Theorem 2.5} Let $R$ be a graded ring with unity and $P$ be a graded ideal of $R$. Then the following assertions are equivalent:
\begin{enumerate}
\item $P$ is a graded almost prime ideal of $R$.
\item If $x, y\in h(R)$ such that $\langle x\rangle\langle y\rangle\subseteq P$ and $\langle x\rangle\langle y\rangle\nsubseteq P^{2}$, then either $x\in P$ or $y\in P$.
\item If $x, y\in h(R)$ such that $xRy\subseteq P$ and $xRy\nsubseteq P^{2}$, then either $x\in P$ or $y\in P$.
\item $(P:\langle x\rangle)=P\bigcup(P^{2}:\langle x\rangle)$, and $(P:_{*}\langle x\rangle)=P\bigcup(P^{2}:_{*}\langle x\rangle)$ for all $x\in h(R)-P$.
\item Either $(P:\langle x\rangle)=P$ or $(P:\langle x\rangle)=(P^{2}:\langle x\rangle)$, and either $(P:_{*}\langle x\rangle)=P$ or $(P:_{*}\langle x\rangle)=(P^{2}:_{*}\langle x\rangle)$ for all $x\in h(R)-P$.
\end{enumerate}
\end{thm}

\begin{proof} $(1)\Rightarrow(2)$: Suppose that $x, y\in h(R)$ such that $\langle x\rangle\langle y\rangle\subseteq P$ and $\langle x\rangle\langle y\rangle\nsubseteq P^{2}$. Then $(RxR)(RyR)=R\langle x\rangle R\langle y\rangle=R\langle x\rangle\langle y\rangle\subseteq RP=P$. If $(RxR)(RyR)\subseteq P^{2}$, then $\langle x\rangle\langle y\rangle\subseteq R\langle x\rangle\langle y\rangle=(RxR)(RyR)\subseteq P^{2}$ which is a contradiction. Hence, $(RxR)(RyR)\nsubseteq P^{2}$. Thus, either $RxR\subseteq P$ or $RyR\subseteq P$ by (1), which implies that either $x\in P$ or $y\in P$.

$(2)\Rightarrow(3)$: Assume that $x, y\in h(R)$ such that $xRy\subseteq P$ and $xRy\nsubseteq P^{2}$. Then $\langle x\rangle\langle y\rangle=xRyR\subseteq PR=P$. If $\langle x\rangle\langle y\rangle\subseteq P^{2}$, then we get $xRy\subseteq xRyR=\langle x\rangle\langle y\rangle\subseteq P^{2}$, a contradiction. So, $\langle x\rangle\langle y\rangle\nsubseteq P^{2}$, and then either $x\in P$ or $y\in P$ by (2).

$(3)\Rightarrow(4)$: Let $x\in h(R)-P$ and $y\in (P:\langle x\rangle)$. Then $y_{g}\in (P:\langle x\rangle)$ for all $g\in G$, and then for $g\in G$, $xRy_{g}\subseteq \langle x\rangle y_{g}\subseteq P$. If $xRy_{g}\nsubseteq P^{2}$, then by (3), we have $y_{g}\in P$ for all $g\in G$ since $x\notin P$, and then $y\in P$. If
$xRy_{g}\subseteq P^{2}$, then $\langle x\rangle y_{g}=RxRy_{g}\subseteq RP^{2}=P^{2}$. Hence, $y_{g}\in (P^{2}:\langle x\rangle)$ for all $g\in G$, which implies that $y\in (P^{2}:\langle x\rangle)$. Therefore, $(P:\langle x\rangle)\subseteq P\bigcup(P^{2}:\langle x\rangle)$. Let $y\in P\bigcup(P^{2}:\langle x\rangle)$. If $y\in P$, then $y_{g}\in P$ for all $g\in G$, and then for $g\in G$, $\langle x\rangle y_{g}\subseteq \langle x\rangle P\subseteq P$, and thus $y_{g}\in (P:\langle x\rangle)$ for all $g\in G$, that gives $y\in (P:\langle x\rangle)$. If $y\in(P^{2}:\langle x\rangle)$, then $y_{g}\in(P^{2}:\langle x\rangle)$ for all $g\in G$, and then for $g\in G$, $\langle x\rangle y_{g}\subseteq P^{2}\subseteq P$, which implies that $y_{g}\in (P:\langle x\rangle)$ for all $g\in G$, which yields that $y\in (P:\langle x\rangle)$. Thus, $P\bigcup(P^{2}:\langle x\rangle)\subseteq (P:\langle x\rangle)$, and consequently, $(P:\langle x\rangle)=P\bigcup(P^{2}:\langle x\rangle)$. Similarly, one can prove that $(P:_{*}\langle x\rangle)=P\bigcup(P^{2}:_{*}\langle x\rangle)$.

$(4)\Rightarrow(5)$: It is understandable.

$(5)\Rightarrow(1)$: Let $X$ and $Y$ be two graded ideals of $R$ such that $XY\subseteq P$. Suppose that $X\nsubseteq P$ and $Y\nsubseteq P$. We show that $XY\subseteq P^{2}$. Firstly, we show that $(X-P)Y\subseteq P^{2}$. Let $x\in X-P$. Then we have $\langle x\rangle Y\subseteq XY\subseteq P$, which implies that $Y\subseteq (P:\langle x\rangle)$, and then by assumption, we have $Y\subseteq (P^{2}:\langle x\rangle)$ since $Y\nsubseteq P$. Therefore, $xY\subseteq \langle x\rangle Y\subseteq P^{2}$. Consequently, $(X-P)Y\subseteq P^{2}$. Secondly, we show that $X(Y-P)\subseteq P^{2}$. Let $y\in Y-P$. Then, $X\langle y\rangle\subseteq XY\subseteq P$, and so $X\subseteq (P:_{*}\langle y\rangle)$, and then by assumption we get $X\subseteq (P^{2}:_{*}\langle y\rangle)$ since $X\nsubseteq P$. Thus, $Xy\subseteq X\langle y\rangle\subseteq P^{2}$, consequently $X(Y-P)\subseteq P^{2}$. The final step that completes the proof is observing that
\begin{center}
$XY=(X-P)Y+(X\bigcap P)(Y-P)+(X\bigcap P)(Y\bigcap P)\subseteq (X-P)Y+X(Y-P)+(X\bigcap P)(Y\bigcap P)\subseteq P^{2}$.
\end{center}
\end{proof}

Examples \ref{2}, \ref{3} and \ref{4} show that a graded almost prime right ideal does not have to be a graded prime right ideal. In the next result, we consider a case such that the concepts above are corresponding.

\begin{thm}\label{Theorem 2.6} Let $R$ be a graded ring with unity and $I$ be a graded right ideal of $R$ with $(I^{2}:I)\subseteq I$. Then $I$ is a graded prime right ideal of $R$ if and only if $I$ is a graded almost prime right ideal of $R$.
\end{thm}

\begin{proof} Suppose that $I$ is a graded almost prime right ideal of $R$ which is not graded prime. Then there exist graded right ideals $X, Y$ of $R$ such that $XY\subseteq I$ with $X\nsubseteq I$ and $Y\nsubseteq I$, and then $XY\subseteq I^{2}$. Let $x\in X-I$ and $y\in Y-I$. Then there exist $g, h\in G$ such that $x_{g}\in X-I$ and $y_{h}\in Y-I$, and then $(x_{g}R+I)y_{h}R=x_{g}Ry_{h}R+Iy_{h}R\subseteq XY+Iy_{h}R\subseteq I$. If $(x_{g}R+I)y_{h}R\subseteq I^{2}$, then $Iy_{h}R\subseteq I^{2}$. This implies that $Iy_{h}\subseteq I^{2}$, and thus $y_{h}\in(I^{2}:I)\subseteq I$, a contradiction. Hence, either $x_{g}R+I\subseteq I$ or $y_{h}R\subseteq I$, which implies that $x_{g}\in I$ or $y_{h}\in I$, a contradiction. The converse is obvious.
\end{proof}

\begin{thm}\label{Theorem 2.7} Let $R$ be graded ring and $I$ be a graded right ideal of $R$ such that $I^{2}=\{0\}$. Then $I$ is a graded weakly prime right ideal of $R$ if and only if $I$ is a graded almost prime right ideal of $R$.
\end{thm}

\begin{proof} It is understandable.
\end{proof}

\begin{rem} In Example \ref{2}, $J=\{0, z\}$ is a graded almost prime right ideal of $R$ and $J^{2}=\{0\}$. Hence, $J$ is a graded weakly prime right ideal of $R$ by Theorem \ref{Theorem 2.7}.
\end{rem}

\begin{cor}\label{Corollary 2.8} Let $R$ be graded ring such that $R^{2}=\{0\}$ and $I$ be a graded right ideal of $R$. Then $I$ is a graded weakly prime right ideal of $R$ if and only if $I$ is a graded almost prime right ideal of $R$.
\end{cor}

\begin{proof} The result is a consequence of Theorem \ref{Theorem 2.7}.
\end{proof}

Let $R$ be a $G$-graded ring and $P$ be a graded ideal of $R$. Then $R/P$ is a $G$-graded ring by $(R/P)_{g}=(R_{g}+P)/P$ for all $g\in G$. It has been proved in (\cite{Jaber Bataineh Khashan}, Theorem 3) that if $R$ is a commutative graded ring, then $I$ is a graded almost prime ideal of $R$ if and only if $I/I^{2}$ is a graded weakly prime ideal of $R/I^{2}$. In the next result, we prove that the same holds for graded almost prime right ideals and graded weakly prime right ideals in a non-commutative graded ring.

\begin{thm}\label{Theorem 2.11} Let $R$ be a graded ring and $I$ be a graded ideal of $R$. Then $I$ is a graded almost prime right ideal of $R$ if and only if $I/I^{2}$ is a graded weakly prime right ideal of $R/I^{2}$.
\end{thm}

\begin{proof} Suppose that $I$ is a graded almost prime right ideal of $R$. Assume that $\overline{X}$ and $\overline{Y}$ are graded right ideals of $R/I^{2}$ such that $\overline{0}\neq \overline{X}\overline{Y}\subseteq \overline{I}=I/I^{2}$. Then there exist graded right ideals $X\supseteq I^{2}$ and $Y\supseteq I^{2}$ of $R$ such that $\overline{X}=X/I^{2}$ and $\overline{Y}=Y/I^{2}$. Therefore, $I^{2}/I^{2}\neq(XY+I^{2})/I^{2}\subseteq I/I^{2}$, hence $I^{2}\neq XY\subseteq I$. So, we have that either $X\subseteq I$ or $Y\subseteq I$ since $XY\nsubseteq I^{2}$. This gives that either $\overline{X}\subseteq \overline{I}$ or $\overline{Y}\subseteq \overline{I}$. Conversely, suppose that $X$ and $Y$ are graded right ideals of $R$ such that $XY\subseteq I$ and $XY\nsubseteq I^{2}$.
Then $\overline{X}=(X+I^{2})/I^{2}$ and $\overline{Y}=(Y+I^{2})/I^{2}$ are graded right ideals of $R/I^{2}$. Moreover, $\overline{X}\hspace{0.1cm}\overline{Y}=(XY+XI^{2}+I^{2}Y+I^{4}+I^{2})/I^{2}\subseteq I/I^{2}=\overline{I}$, and $\overline{X}\hspace{0.1cm}\overline{Y}\nsubseteq\overline{I^{2}}$. Thus, $\overline{0}\neq\overline{X}\hspace{0.1cm}\overline{Y}\subseteq \overline{I}$ and hence either $\overline{X}\subseteq \overline{I}$ or $\overline{Y}\subseteq\overline{I}$. Consequently, $X\subseteq I$ or $Y\subseteq I$.
\end{proof}

Let $R$ and $T$ be two $G$-graded rings. Then a ring homomorphism $f:R\rightarrow T$ is said to be a graded homomorphism if $f(R_{g})\subseteq T_{g}$ for all $g\in G$ \cite{Nastasescu}.

\begin{thm}\label{Theorem 2.12} Let $f:R\rightarrow T$ be a graded ring epimorphism and $I$ be a graded almost prime right ideal of $R$ such that $Ker(f)\subseteq I$. Then $f(I)$ is a graded almost prime right ideal of $T$.
\end{thm}

\begin{proof} Suppose that $CD\subseteq f(I)$ and $CD\nsubseteq (f(I))^{2}$ for graded right ideals $C$ and $D$ of $T$. Then $Ker(f)\subseteq f^{-1}(C)=X$ and $Ker(f)\subseteq f^{-1}(D)=Y$ are graded right ideals of $R$. Hence, $f(X)=C$ and $f(Y)=D$ since $f$ is an epimorphism, and then we have that $f(XY)=CD\subseteq f(I)$, and $f(XY)\nsubseteq (f(I))^{2}=f(I^{2})$. Thus $XY\subseteq f^{-1}(f(XY))\subseteq f^{-1}(f(I))=I$ and $XY\nsubseteq I^{2}$. So, either $X\subseteq I$ or $Y\subseteq I$, and so either $C\subseteq f(I)$ or $D\subseteq f(I)$.
\end{proof}

\begin{cor}\label{Corollary 2.13} Let $f:R\rightarrow T$ be a graded ring epimorphism and $J$ be a graded right ideal of $T$ such that $f^{-1}(J)$ is a graded almost prime right ideal of $R$. Then $J$ is a graded almost prime right ideal of $T$.
\end{cor}

\begin{proof} Since $f^{-1}(J)$ is a graded right ideal of $R$ and $Ker(f)\subseteq f^{-1}(J)$, the result follows by Theorem \ref{Theorem 2.12}.
\end{proof}

\begin{thm}\label{Theorem 2.14} Let $f:R\rightarrow T$ be a graded ring epimorphism and $I$ be a graded right ideal of $R$ such that $Ker(f)\subseteq I^{2}$. If $f(I)$ is a graded almost prime right ideal of $T$, then $I$ is a graded almost prime right ideal of $R$.
\end{thm}

\begin{proof} Suppose that $XY\subseteq I$ and $XY\nsubseteq I^{2}$ for graded right ideals $X$ and $Y$ of $R$. Then $f(X)f(Y)=f(XY)\subseteq f(I)$. If $f(XY)\subseteq f(I^{2})$, then $XY\subseteq f^{-1}(f(XY))\subseteq f^{-1}(f(I^{2}))=I^{2}$, a contradiction. Hence, $f(X)f(Y)=f(XY)\nsubseteq (f(I))^{2}$. Since $f(I)$ is a graded almost prime right ideal of $T$, either $f(X)\subseteq f(I)$ or $f(Y)\subseteq f(I)$. Therefore, either $X\subseteq f^{-1}(f(X))\subseteq f^{-1}(f(I))=I$ or $Y\subseteq I$.
\end{proof}

\begin{cor}\label{Corollary 2.15} Let $f:R\rightarrow T$ be a graded ring epimorphism and $J$ be a graded almost prime right ideal of $T$ such that $Ker(f)\subseteq(f^{-1}(J))^{2}$. Then $f^{-1}(J)$ is a graded almost prime right ideal of $R$.
\end{cor}

\begin{proof} Assume that $I=f^{-1}(J)$. Then $I$ is a graded almost prime right ideal of $R$ by Theorem \ref{Theorem 2.14} since $Ker(f)\subseteq I^{2}$ and $f(I)=f(f^{-1}(J))=J$ is a graded almost prime right ideal of $T$.
\end{proof}

\begin{thm}\label{Theorem 2.16} Let $R$ be a graded ring and $K$ be a graded ideal of $R$. Suppose that $P$ is a graded right ideal of $R$ such that $K\subseteq P$. If $P$ is a graded almost prime right ideal of $R$, then $P/K$ is a graded almost prime right ideal of $R/K$.
\end{thm}

\begin{proof} Suppose that $\overline{X}\hspace{0.1cm}\overline{Y}\subseteq \overline{P}=P/K$ and $\overline{X}\hspace{0.1cm}\overline{Y}\nsubseteq \overline{P}^{2}$ for graded right ideals $\overline{X}$ and $\overline{Y}$ of $R/K$. Assume that $\overline{X}=X/K$ and $\overline{Y}=Y/K$ for some graded right ideals $X\supseteq K$ and $Y\supseteq K$. Then $(XY+K)/K\subseteq P/K$ and $(XY+K)/K\nsubseteq(P^{2}+K)/K$ which implies that $XY\subseteq P$ and $XY\nsubseteq P^{2}$. So, either $X\subseteq P$ or $Y\subseteq P$, and hence $\overline{X}\subseteq \overline{P}$ or $\overline{Y}\subseteq \overline{P}$.
\end{proof}

\begin{rem} One can choose a graded ring $R$ and a graded ideal $K$ of $R$ that is not a graded almost prime right ideal of $R$. Indeed, $\overline{0}=K/K$ is a graded almost prime right ideal of $R/K$. So, the converse of Theorem \ref{Theorem 2.16} is not true in general.
\end{rem}

In Example \ref{2}, every graded right ideal of $R$ is graded almost prime. In the rest of our article, we consider some cases where all graded right ideals of a non-commutative graded ring are graded almost prime. The next result is a repercussion of Theorem \ref{Theorem 2.7}.

\begin{cor}\label{Corollary 3.2} Let $R$ be a graded ring such that $I^{2}=\{0\}$ for every graded right ideal $I$ of $R$. Then every graded right ideal of $R$ is graded almost prime if and only if every graded right ideal of $R$ is graded weakly prime.
\end{cor}

The next result is a consequence of Corollary \ref{Corollary 2.8}.

\begin{cor}\label{Remark 3.3} Let $R$ be a graded ring such that $R^{2}=\{0\}$. Then every graded right ideal of $R$ is graded almost prime if and only if every graded right ideal of $R$ is graded weakly prime.
\end{cor}

\begin{thm}\label{Theorem 3.4} Let $f:R\rightarrow T$ be a graded ring epimorphism. If every graded right ideal of $R$ is graded almost prime, then so is $T$.
\end{thm}

\begin{proof} Let $I$ be a graded right ideal of $T$. Then $f^{-1}(I)\supseteq Ker(f)$ is a graded almost prime right ideal of $R$, and then by Theorem \ref{Theorem 2.12} we get that $f(f^{-1}(I))=I$ is a graded almost prime right ideal of $T$.
\end{proof}

\begin{thm}\label{Theorem 3.5} Let $f:R\rightarrow T$ be a graded ring epimorphism such that $Ker(f)\subseteq P^{2}$ for any graded right ideal $P$ of $R$. If every graded right ideal of $T$ is graded almost prime, then so is $R$.
\end{thm}

\begin{proof} Let $I$ be a graded right ideal of $R$. Then $f(I)$ is a graded almost prime right ideal of $T$, and then by Theorem \ref{Theorem 2.14} we get that $I$ is a graded almost prime right ideal of $R$.
\end{proof}

\begin{thm}\label{Theorem 3.6} Let $R$ be a graded ring and $P$ be a graded ideal of $R$. If every graded right ideal of $R$ is graded almost prime, then so is $R/P$.
\end{thm}

\begin{proof} Suppose that $\overline{I}$ is a graded right ideal of $R/P$. Then there exists a graded right ideal $I\supseteq P$ of $R$ such that $\overline{I}= I/P$. Clearly, $I$ is a graded almost prime right ideal of $R$. Hence, by Theorem \ref{Theorem 2.16} $\overline{I}$ is a graded almost prime right ideal of $R/P$.
\end{proof}

\end{document}